\documentclass[10pt]{amsart}
\usepackage{amsfonts}
\usepackage{ifthen}
\usepackage{amsthm}
\usepackage{amsmath}
\usepackage{graphicx}
\usepackage{amscd,amssymb,amsthm}
\usepackage{graphicx}
\newcommand{\real}{\operatorname{Re}}

\newcounter{minutes}
\divide\time by 60
\newcounter{hours}
\multiply\time by 60 \addtocounter{minutes}{-\time}

\theoremstyle{plain}

\newtheorem{theorem}{Theorem}[section]
\newtheorem{corollary}[theorem]{Corollary}
\newtheorem{definition}{Definition}[section]
\newtheorem{lemma}{Lemma}[section]
\numberwithin{equation}{section}

\begin{document}

\def\thefootnote{}
\footnotetext{ \texttt{File:~\jobname .tex,
          printed: \number\year-\number\month-0\number\day,
          \thehours.\ifnum\theminutes<10{0}\fi\theminutes}
} \makeatletter\def\thefootnote{\@arabic\c@footnote}\makeatother

\title[Differential subordinations involving generalized Bessel functions]{Differential subordinations involving generalized Bessel functions}
\author{\'Arp\'ad Baricz}
\address{Department of Economics, Babe\c{s}-Bolyai University, Cluj-Napoca
400591, Romania}
\email{bariczocsi@yahoo.com}
\author{Erhan Deniz}
\address{Department of Mathematics, Faculty of Science and Letters, Kafkas
University, 36100, Kars, Turkey.}
\email{edeniz36@gmail.com}
\author{Murat \c{C}a\u{g}lar}
\address{Department of Mathematics, Faculty of Science, Atat\"urk University,
25240, Erzurum, Turkey.}
\email{mcaglar25@gmail.com}
\author{Halit Orhan}
\address{Department of Mathematics, Faculty of Science, Atat\"urk University,
25240, Erzurum, Turkey.}
\email{orhanhalit607@gmail.com}

\begin{abstract}
In this paper our aim is to present some subordination and superordination results, by using an operator, which
involves the normalized form of the generalized Bessel functions of first kind. These results are obtained by
investigating some appropriate classes of admissible functions. We obtain also some sandwich-type results and we point out various known
or new special cases of our main results.
\end{abstract}

\maketitle

\section{\noindent \textbf{Introduction, definitions and preliminaries}}

Let $\mathcal{H}\left( \mathrm{{\mathbb{D}}}\right) $ be the class of analytic
functions in $\mathrm{{\mathbb{D}}}=\left\{ {z\in \mathrm{\mathbb{C}}:\left\vert z\right\vert <1}\right\} $ and $\mathcal{H[}a,n]$ be the
subclass of $\mathcal{H}\left( \mathrm{{\mathbb{D}}}\right) $ consisting of
functions of the form $f(z)=a+a_{n}z^{n}+a_{n+1}z^{n+1}+{\dots},$ with $\mathcal{%
H}_{0}=\mathcal{H[}0,1]$ and $\mathcal{H}_{1}=\mathcal{H[}1,1].$ We denote
by $\mathcal{A}$ the class of all functions of the form
\begin{equation}
f(z)=z+\sum\limits_{n\geq1}{a_{n+1}z^{n+1}},  \label{1}
\end{equation}
which are analytic in the open unit disk $\mathrm{{\mathbb{D}}}.$ Let $f$ and $F$
be members of $\mathcal{H}\left( \mathrm{{\mathbb{D}}}\right) .$ The
function $f$ is said to be \textit{subordinate} to $F,$ or $F$ is said to be
\textit{superordinate} to $f,$ written symbolically as%
\begin{equation*}
f\prec F\text{\ \ \ }\mathrm{or}\ \ \ \ f(z)\prec F(z)\ \ \ \left( z\in
\mathrm{{\mathbb{D}}}\right) ,
\end{equation*}

\noindent if there exists a function $w(z)$ analytic in $\mathrm{{\mathbb{D}}}$ (with $w(0)=0$ and $\left\vert w(z)\right\vert <1$ in $\mathrm{{\mathbb{D}}}$) such that $f(z)=F(w(z)),\;z\in \mathrm{{\mathbb{D}}}.$ In
particular, if the function $F$ is univalent in $\mathrm{{\mathbb{D}}},$
then we have the following equivalence

\begin{equation*}
f(z)\prec F(z)\ \ (z\in \mathrm{{\mathbb{D}}})\ \ \ \Longleftrightarrow \ \ \
f(0)=F(0)\ \ \ \mathrm{and}\ \ f(\mathrm{{\mathbb{D}}})\subset F(\mathrm{{%
\mathbb{D}}}).
\end{equation*}

Let us consider the following second-order linear homogenous differential
equation (for more details see \cite{Ba2})
\begin{equation}
z^{2}\omega''(z)+bz\omega'(z)+\left[
cz^{2}-p^{2}+(1-b)p\right] \omega (z)=0,\text{ \ \ }\left( b,c,p\in\mathbb{C}\right).  \label{2}
\end{equation}
The function $\omega _{p,b,c},$ which is called the generalized Bessel
function of the first kind of order $p,$ is defined as a particular solution of
$\left( \ref{2}\right) .$ The function $\omega _{p,b,c}$ has the familiar
representation as follows
\begin{equation}
\omega _{p,b,c}(z)=\sum_{n\geq0}\frac{(-c)^{n}}{n!\Gamma \left(
p+n+\frac{b+1}{2}\right) }\left( \frac{z}{2}\right) ^{2n+p},\text{\ \ \ }z\in\mathbb{C},  \label{3}
\end{equation}
where $\Gamma $ stands for the Euler gamma function and $\kappa =p+\frac{b+1}{2}\notin\mathbb{Z}_{0}^{-}=\{0,-1,-2,\dots\}.$ The series $\left( \ref%
{3}\right) $ permits the study of Bessel, modified Bessel, spherical
Bessel, modified spherical Bessel and ultraspherical Bessel functions all together. It is worth mentioning that, in particular,

$\bullet $ for $b=c=1$ in $\left( \ref{3}\right) ,\;$we obtain the familiar
Bessel function of the first kind of order $p$ defined by (see \cite{Wa};
see also \cite{Ba2}) $\;$%
\begin{equation}
J_{p}(z)=\sum_{n\geq0}\frac{(-1)^{n}}{n!\Gamma \left( p+n+1\right) }%
\left( \frac{z}{2}\right) ^{2n+p},\text{\ \ \ }z\in\mathbb{C}
,  \label{4}
\end{equation}

$\bullet $ for $b=1$ and $c=-1$ in $\left( \ref{3}\right) ,\;$ we obtain the
modified Bessel function of the first kind of order $p$ defined by (see \cite{Wa}; see also \cite{Ba2})
\begin{equation}
I_{p}(z)=\sum_{n\geq0}\frac{1}{n!\Gamma \left( p+n+1\right) }\left(
\frac{z}{2}\right) ^{2n+p},\text{\ \ \ }z\in
\mathbb{C},  \label{5}
\end{equation}

$\bullet $ for $b=2$ and $c=1$ in $\left( \ref{3}\right) ,\;$the function $\omega _{p,b,c}$ reduces to $\sqrt{2}j_{p}\diagup \sqrt{\pi }$ where $j_{p}$ is the spherical Bessel function of the first kind of order $p,$
defined by (see \cite{Ba2})
\begin{equation}
j_{p}(z)=\sqrt{\frac{\pi }{2}}\sum_{n\geq0}\frac{(-1)^{n}}{n!\Gamma
\left( p+n+\frac{3}{2}\right) }\left( \frac{z}{2}\right) ^{2n+p},\text{\ \ \ }z\in\mathbb{C}.  \label{6}
\end{equation}

Now, we consider the function $\varphi _{p,b,c}:\mathbb{D}\to\mathbb{C},$ defined
in terms of the generalized Bessel function $\omega _{p,b,c},$ by the
transformation
\begin{equation}
\varphi _{p,b,c}(z)=2^{p}\Gamma \left( p+\frac{b+1}{2}\right)
z^{1-\frac{p}{2}}\omega _{p,b,c}(\sqrt{z}).  \label{7}
\end{equation}
By using the well-known Pochhammer symbol (or the \textit{shifted}
factorial) $(\lambda )_{\mu }$ defined, for $\lambda ,\mu \in\mathbb{C}$ and in terms of the Euler $\Gamma -$function, by
\begin{equation*}
(\lambda )_{\mu }=\frac{\Gamma (\lambda +\mu )}{\Gamma (\lambda )}=\left\{
\begin{array}{c}
1,\text{ \ \ \ \ \ \ \ \ \ \ \ \ \ \ \ \ \ \ \ \ \ } \mu =0,\text{ }
\lambda \in\mathbb{C}\setminus \{0\} \\
\lambda (\lambda +1)\dots(\lambda +n-1),\text{\ \ \ } \mu =n\in\mathbb{N},\text{ }\lambda \in\mathbb{C}
\end{array},
\right.
\end{equation*}
it being understood conventionally that $(0)_{0}=1$, we obtain the
following series representation for the function $\varphi _{p,b,c}$ given
by $\left( \ref{7}\right) $
\begin{equation}
\varphi _{p,b,c}(z)=z+\sum_{n\geq1}\frac{(-c)^{n}}{4^{n}\left( \kappa
\right) _{n}}\frac{z^{n+1}}{n!},  \label{8}
\end{equation}
where $\kappa=p+\frac{b+1}{2}\notin\mathbb{Z}_{0}^{-}$ and $\mathbb{N}=\{1,2,3,\dots\}.$ For convenience, we write $\varphi _{\kappa
,c}(z)=\varphi _{p,b,c}(z).$

For $f\in \mathcal{A}$ given by (\ref{1}) and $g$ given by $g(z)=z+\sum_{n\geq1}b_{n+1}z^{n+1},$ the Hadamard product\ (or
convolution) of $f$ and $g$ is defined by
\begin{equation*}
(f\ast g)(z)=z+\sum_{n\geq1}a_{n+1}b_{n+1}z^{n+1}=(g\ast f)(z), \text{ \ }z\in \mathrm{{\mathbb{D}}}.
\end{equation*}
Note that $f\ast g\in \mathcal{A}.$ Now, we consider the $B_{\kappa }^{c}-$operator, which is defined as follows
\begin{eqnarray}\label{11}
B_{\kappa }^{c}f(z) =\varphi _{\kappa ,c}(z)\ast f(z)=z+\sum_{n\geq1}\frac{(-c)^{n}a_{n+1}}{4^{n}\left( \kappa \right)_{n}}\frac{z^{n+1}}{n!}
\end{eqnarray}
We note that by using the definition (\ref{11}) we obtain that
\begin{equation}
z\left[ B_{\kappa +1}^{c}f(z)\right] ^{\prime }=\kappa B_{\kappa
}^{c}f(z)-(\kappa -1)B_{\kappa +1}^{c}f(z)  \label{12}
\end{equation}%
where $\kappa =p+\frac{b+1}{2}\notin\mathbb{Z}_{0}^{-}.$ It is worth to mention that in fact $B_{\kappa }^{c}f$ given by (\ref{11}) is an
elementary transform of the generalized hypergeometric function, that is, we have $$B_{\kappa }^{c}f(z)=z{}_{0}F_{1}\left(~\kappa ;\frac{-c}{4}z\right)\ast
f(z).$$ In particular, for the $B_{\kappa }^{c}-$operator we obtain the following operators:

$\bullet $ Choosing $b=c=1$ in (\ref{11}) or (\ref{12}) we obtain the
operator $\mathcal{J}_{p}:\mathcal{A}\mathbb{\rightarrow }\mathcal{A}$ \
related with Bessel function, defined by
\begin{equation}
\mathcal{J}_{p}f(z)=\varphi _{p,1,1}(z)\ast f(z)=\left[ 2^{p}\Gamma \left(
p+1\right) z^{1-\frac{p}{2}}J_{p}(\sqrt{z})\right] \ast f(z)=z+\sum_{n\geq1}
\frac{(-1)^{n}a_{n+1}}{4^{n}\left( p+1\right) _{n}}\frac{z^{n+1}}{n!},
\label{121}
\end{equation}
which satisfies the recurrence relation
\begin{equation*}
z\left[ \mathcal{J}_{p+1}f(z)\right]'=(p+1)\mathcal{J}_{p}f(z)-p
\mathcal{J}_{p+1}f(z).
\end{equation*}
$\bullet $ Choosing $b=1$ and $c=-1$ in (\ref{11}) or (\ref{12}) we obtain
the operator $\mathcal{I}_{p}:\mathcal{A}\mathbb{\rightarrow }\mathcal{A}$ \
related with modified Bessel function, defined by
\begin{equation}
\mathcal{I}_{p}f(z)=\varphi _{p,1,-1}(z)\ast f(z)=\left[ 2^{p}\Gamma \left(
p+1\right) z^{1-\frac{p}{2}}I_{p}(\sqrt{z})\right] \ast f(z)=z+\sum_{n\geq1}
\frac{a_{n+1}}{4^{n}\left( p+1\right) _{n}}\frac{z^{n+1}}{n!},  \label{122}
\end{equation}
which satisfies the recursive relation
\begin{equation*}
z\left[ \mathcal{I}_{p+1}f(z)\right]'=(p+1)\mathcal{I}_{p}f(z)-p%
\mathcal{I}_{p+1}f(z).
\end{equation*}
$\bullet $ Choosing $b=2$ and $c=1$ in (\ref{11}) or (\ref{12}) we obtain
the operator $\mathcal{S}_{p}:\mathcal{A}\mathbb{\rightarrow }\mathcal{A}$ \
related with spherical Bessel function, defined by
\begin{equation}
\mathcal{S}_{p}f(z)=\left[ \pi ^{-\frac{1}{2}}2^{p+\frac{1}{2}}\Gamma \left( p+\frac{3}{2}\right) z^{1-\frac{p}{2}}j_{p}(\sqrt{z})\right] \ast f(z)=z+\sum_{n\geq1}
\frac{(-1)^{n}a_{n+1}}{4^{n}\left( p+\frac{3}{2}\right) _{n}}\frac{z^{n+1}}{n!},  \label{123}
\end{equation}
which satisfies the recurrence relation
\begin{equation*}
z\left[ \mathcal{S}_{p+1}f(z)\right]'=\left(p+\frac{3}{2}\right)\mathcal{S}
_{p}f(z)-\left(p+\frac{1}{2}\right)\mathcal{S}_{p+1}f(z).
\end{equation*}

For further result on the transformation \eqref{8} of the generalized Bessel function
we refer to the recent papers (see \cite{An,Ba1,Ba2,Ba3,Pr,szasz,szasz2}), where
among other things some interesting functional inequalities, integral
representations, application of admissible functions, extensions of some
known trigonometric inequalities, starlikeness and convexity and univalence
were established. Most of these results were motivated by the research on geometric
properties of Gaussian and Kummer hypergeometric functions. For more details we refer to the papers \cite{Mil3,Ow,Po1,Po2,Po3}
and to the references therein. We also mention that very recently, Deniz et al. \cite{De} and Deniz \cite{De1}
were interested on the univalence and convexity of some integral operators defined
by normalized form of the generalized Bessel functions of the first kind
given by (\ref{8}).

In the present paper, by making use of the differential subordination and
differential superordination results of Miller and Mocanu \cite{Mil1,Mil2},
we determine certain classes of admissible functions and obtain some subordination
and superordination implications of analytic functions associated with the $B_{\kappa }^{c}-$%
operator defined by (\ref{11}), together with some sandwich-type
theorems.

To prove our main results, we need the following definitions and theorems.

\begin{definition}\cite[p. 21]{Mil1}
Let us denote by $\mathcal{Q}$ the class of functions $q,$ which are analytic and
injective on $\overline{\mathbb{D}}\setminus E(q),$ where
\begin{equation*}
E(q)=\left\{ \zeta :\zeta \in \partial \mathrm{{\mathbb{D}}}\text{ and }%
\underset{z\rightarrow \zeta }{\lim }q(z)=\infty \right\} ,
\end{equation*}
and are such that $q'(\zeta )\neq 0$ $(\zeta \in \partial \mathbb{D}\setminus E(q)).$ Further let the subclass of $\mathcal{Q}$ for which $%
q(0)=a$ be denoted by $\mathcal{Q}(a),$ $\mathcal{Q}(0)\equiv \mathcal{Q}%
_{0} $ and $\mathcal{Q}(1)\equiv \mathcal{Q}_{1}.$
\end{definition}

\begin{definition}\cite[p. 27]{Mil1}\label{d1}
Let $\Omega $ be a set in $\mathbb{C},$ $q\in \mathcal{Q}$ and $n\in\mathbb{N}.$ The class of admissible functions $\Psi _{n}\left[ \Omega ,q\right] $
consists of those functions $\psi :\mathbb{C}^{3}\times \mathrm{{\mathbb{D}}}\rightarrow\mathbb{C}$ that satisfy the following admissibility condition
\begin{equation*}
\psi (r,s,t;z)\notin \Omega
\end{equation*}
whenever
\begin{equation*}
r=q(\zeta ),\text{ \ \ }s=k\zeta q'(\zeta )\text{ \ \ and \ \ }
\real\left( \frac{t}{s}+1\right) \geq k\real\left( \frac{\zeta q''(\zeta )}{q'(\zeta )}+1\right)
\end{equation*}
\begin{equation*}
\left( z\in \mathrm{{\mathbb{D}}};\;\zeta \in \partial \mathrm{{\mathbb{D}}
}\setminus E(q);\;k\geq n\right) .
\end{equation*}
We write $\Psi _{1}\left[ \Omega ,q\right] $ simply as $\Psi \left[ \Omega ,q
\right] .$
\end{definition}

In particular when $$q(z)=M\frac{Mz+a}{M+\overline{a}z},$$ with $M>0$ and
$\left\vert a\right\vert <M,$ then $q(\mathbb{D})=\mathbb{D}_{M}=\left\{
w:\left\vert w\right\vert <M\text{ }\right\} ,$ $q(0)=a,$ $E(q)=\varnothing $
and $q\in \mathcal{Q}.$ In this case, we set $\Psi _{n}[\Omega,M,a]=\Psi _{n}[\Omega,q]$, and in the special case when the set $\Omega=\mathbb{D}_{M}$, the class is simply denoted by $\mathbb{D}_{M}[M,a].$

\begin{definition}\cite[p. 817]{Mil2} \label{d2} Let $\Omega $ be a set in $\mathbb{C}$ and $q\in \mathcal{H}\left[ a,n\right] $ with $q^{\prime }(z)\neq 0.$ The class of admissible functions $\Psi _{n}'\left[ \Omega ,q\right] $
consist of those functions
\begin{equation*}
\psi :\mathbb{C}^{3}\times \overline{\mathrm{{\mathbb{D}}}}\rightarrow\mathbb{C}
\end{equation*}
that satisfy the following admissibility condition
\begin{equation*}
\psi (r,s,t;\varsigma )\in \Omega
\end{equation*}
whenever
\begin{equation*}
r=q(z),\text{ \ \ }s=\frac{zq'(z)}{m}\text{ \ \ and \ \ }\real\left( \frac{t}{s}+1\right) \leq \frac{1}{m}\real\left(
\frac{zq''(z)}{q'(z)}+1\right)
\end{equation*}
\begin{equation*}
\left( z\in\mathbb{D};\;\varsigma \in \partial\mathbb{D};\;m\geq n\geq 1\right) .
\end{equation*}
We write $\Psi _{1}'\left[ \Omega ,q\right] $ simply as $\Psi'\left[ \Omega ,q\right] .$
\end{definition}

For the above two classes of admissible functions, Miller and Mocanu \cite{Mil1,Mil2} proved the following results.

\begin{lemma}\cite[p. 28]{Mil1} \label{l1}
Let $\psi \in \Psi _{n}\left[ \Omega ,q\right] $ with $q(0)=a.$ If the analytic
function ${p}\in\mathcal{H}[a,n]$ satisfies the following inclusion relationship
\begin{equation*}
\psi \left({p}(z),z{p}'(z),z^{2}{p}''(z);z\right) \in \Omega ,
\end{equation*}
for all $z\in \mathbb{D},$ then ${p}\prec q.$
\end{lemma}

\begin{lemma}\cite[p. 818]{Mil2}\label{l2}
Let $\psi \in \Psi _{n}'\left[ \Omega ,q\right] $ with $q(0)=a.$ If ${p}\in \mathcal{Q}(a)$ and $\psi \left({p}(z),z{p}'(z),z^{2}{p}''(z);z\right) $ \ is
univalent in $\mathbb{D},$ then
\begin{equation*}
\Omega \subset \left\{ \psi \left({p}(z),z{p}'(z),z^{2}{p}''(z);z\right) :z\in \mathbb{D}\right\}
\end{equation*}
implies that ${p}\prec q.$
\end{lemma}

\section{\textbf{Subordination results involving the }$B_{\protect\kappa %
}^{c}-$\textbf{operator}}

We begin this section by proving a differential subordination theorem
involving the $B_{\kappa }^{c}-$operator defined by (\ref{11}). In the sequel the
parameter $c$ is an arbitrary complex number, while for the parameter $\kappa$
we will have some assumptions in some cases. To prove our main results,
we need first the following class of admissible functions.

\begin{definition}
\label{d2.1} Let $\Omega$ be a set in $\mathbb{C}$ and $q\in \mathcal{Q}_{0}\cap \mathcal{H}_{0\text{ }}.$ The class of
admissible functions $\Phi _{H}[\Omega ,q]$ consists of those functions $\phi:\mathbb{C}^{3}\times \mathbb{D}\rightarrow\mathbb{C}$ that satisfy the admissibility condition
\begin{equation*}
\phi (u,v,w;z)\notin \Omega
\end{equation*}
whenever
\begin{equation*}
u=q(\zeta ),\ v=\frac{k\zeta q'(\zeta )+(\kappa -1)q(\zeta )}{\kappa }\text{ \ \ }\left( \kappa \in\mathbb{C},\text{ }\kappa \neq 0,1\right)
\end{equation*}
and
\begin{equation*}
\real\left( \frac{\kappa (\kappa -1)w-(\kappa -2)(\kappa -1)u}{v\kappa -(\kappa -1)u}-(2\kappa -3)\right) \geq k\real\left(
\frac{\zeta q''(\zeta )}{q'(\zeta )}+1\right)
\end{equation*}
\begin{equation*}
\left( z\in\mathbb{D};\;\zeta \in \partial\mathbb{D}\setminus E(q);\;k\geq 1\right) .
\end{equation*}
\end{definition}

\begin{theorem}
\label{T2.1} Let $\phi \in \Phi _{H}[\Omega ,q].$ If $f\in \mathcal{A}$
satisfies the following inclusion relationship
\begin{equation}
\left\{ \phi \left( B_{\kappa +1}^{c}f(z),B_{\kappa }^{c}f(z),B_{\kappa
-1}^{c}f(z);z\right) :z\in \mathbb{D}\right\} \subset \Omega ,
\label{2.1}
\end{equation}
then
\begin{equation*}
B_{\kappa +1}^{c}f(z)\prec q(z)\text{ \ \ }(z\in\mathbb{D}).
\end{equation*}
\end{theorem}

\begin{proof}
We define the analytic function ${p}$ in $\mathbb{D}$ by
\begin{equation}
{p}(z)=B_{\kappa +1}^{c}f(z).  \label{2.2}
\end{equation}
Then, differentiating (\ref{2.2}) with respect to $z$ and using the
recursive relation (\ref{12}), we have
\begin{equation}
B_{\kappa }^{c}f(z)=\frac{z{p}'(z)+(\kappa -1){p}(z)}{\kappa }.  \label{2.3}
\end{equation}
Further computations show that
\begin{equation}
B_{\kappa -1}^{c}f(z)=\frac{z^2{p}''(z)+2(\kappa -1)z{p}'(z)+(\kappa -1)(\kappa -2)p(z)}{\kappa
(\kappa -1)}.  \label{2.4}
\end{equation}
We now define the transformations from $\mathbb{C}^{3}$ to $\mathbb{C}$ by
$$u=r,\text{ \ }v=\frac{s+(\kappa -1)r}{\kappa }\text{ \ and \ }w=\frac{t+2(\kappa -1)s+(\kappa -1)(\kappa -2)r}{\kappa (\kappa -1)}.$$
Let
\begin{equation}
\psi (r,s,t;z)=\phi (u,v,w;z)=\phi \left( r,\frac{s+(\kappa -1)r}{\kappa },
\frac{t+2(\kappa -1)s+(\kappa -1)(\kappa -2)r}{\kappa (\kappa -1)};z\right) .
\label{2.6}
\end{equation}
Using equations (\ref{2.2})-(\ref{2.4}), and from (\ref{2.6}), we get
\begin{equation}
\psi ({p}(z),z{p}'(z),z^{2}{p}''(z);z)=\phi \left( B_{\kappa +1}^{c}f(z),B_{\kappa
}^{c}f(z),B_{\kappa -1}^{c}f(z);z\right) .  \label{2.7}
\end{equation}
Hence (\ref{2.1}) assumes the following form
\begin{equation*}
\psi ({p}(z),z{p}'(z),z^{2}{p}''(z);z)\in \Omega .
\end{equation*}
The proof is completed if it can be shown that the admissibility condition
for $\phi \in \Phi _{H}[\Omega ,q]$ is equivalent to the admissibility
condition for as given in Definition \ref{d1}. Note that
\begin{equation*}
\frac{t}{s}+1=\frac{\kappa (\kappa -1)w-(\kappa -2)(\kappa -1)u}{v\kappa
-(\kappa -1)u}-(2\kappa -3),
\end{equation*}
and hence $\psi \in \Psi \lbrack \Omega ,q].$ By Lemma \ref{l1}, we have%
\begin{equation*}
{p}(z)\prec q(z)\text{ \ or \ }B_{\kappa +1}^{c}f(z)\prec q(z)
\end{equation*}
which completes the proof of Theorem \ref{T2.1}.
\end{proof}

If \ $\Omega \neq\mathbb{C}$ is a simply connected domain, then $\Omega =h(\mathbb{D})$ for some
conformal mapping $h$ of $\mathbb{D}$ onto $\Omega .$ In this case the
class $\Phi _{H}[h(\mathbb{D}),q]$ is written as $\Phi _{H}[h,q].$

The following result is an immediate consequence of Theorem \ref{T2.1}.

\begin{corollary}
\label{T2.2} Let $\phi \in \Phi _{H}[h,q].$ If $f\in \mathcal{A}$ satisfies%
\begin{equation}
\phi \left( B_{\kappa +1}^{c}f(z),B_{\kappa }^{c}f(z),B_{\kappa
-1}^{c}f(z);z\right) \prec h(z),  \label{2.8}
\end{equation}%
then%
\begin{equation*}
B_{\kappa +1}^{c}f(z)\prec q(z)\text{ \ \ }(z\in \mathrm{{\mathbb{D}}}).
\end{equation*}
\end{corollary}

Our next result is an extension of Theorem \ref{T2.1} to the case when the
behavior of $q$ on $\partial\mathbb{D}$ is not known.

\begin{corollary}
\label{C2.1} Let $\Omega \subset\mathbb{C}$ and let $q$ be univalent in $\mathbb{D}$ with $q(0)=0.$ Let
$\phi \in \Phi _{H}[\Omega ,q_{\rho }]$ for some $\rho \in (0,1),$ where $q_{\rho }(z)=q(\rho z).$ If $f\in \mathcal{A}$ satisfies the following
inclusion relationship
\begin{equation*}
\phi \left( B_{\kappa +1}^{c}f(z),B_{\kappa }^{c}f(z),B_{\kappa
-1}^{c}f(z);z\right) \in \Omega ,
\end{equation*}
then
\begin{equation*}
B_{\kappa +1}^{c}f(z)\prec q(z)\text{ \ \ }(z\in\mathbb{D}).
\end{equation*}
\end{corollary}

\begin{proof}
We note from Theorem \ref{T2.1} that
\begin{equation*}
B_{\kappa +1}^{c}f(z)\prec q_{\rho }(z)\text{ \ \ }(z\in\mathbb{D}).
\end{equation*}
The result asserted by Corollary \ref{C2.1} is now deduced from the
following subordination relationship
\begin{equation*}
q_{\rho }(z)\prec q(z)\text{ \ \ }(z\in \mathbb{D}).
\end{equation*}
\end{proof}

\begin{theorem}
\label{T2.3} Let $h$ and $q$ be univalent in $\mathbb{D}$
with $q(0)=0$ and set $q_{\rho }(z)=q(\rho z)$ and\ $h_{\rho }(z)=h(\rho z).$
Let $\phi:\mathbb{C}^{3}\times \mathbb{D}\rightarrow\mathbb{C}$ satisfy one of the following conditions
\begin{enumerate}
\item $\phi \in \Phi _{H}[\Omega ,q_{\rho }]$ for some $\rho \in (0,1),$ or
\item there exists $\rho _{0}\in (0,1)$ such that $\phi \in \Phi
_{H}[h_{\rho },q_{\rho }]$ for all $\rho \in (\rho _{0},1).$
\end{enumerate}

If $f\in \mathcal{A}$ satisfies (\ref{2.8}), then
\begin{equation*}
B_{\kappa +1}^{c}f(z)\prec q(z)\text{ \ \ }(z\in\mathbb{D}).
\end{equation*}
\end{theorem}

\begin{proof}
The proof is similar to the proof of \cite[Theorem 2.3d]{Mil1} and
is therefore omitted.
\end{proof}

The next theorem yields the best dominant of the differential subordination (\ref{2.8}).

\begin{theorem}
\label{T2.4} Let $h$ be univalent in $\mathbb{D}$ and $\phi:\mathbb{C}^{3}\times \mathbb{D}\rightarrow\mathbb{C}.$ Suppose that the differential equation
\begin{equation}
\phi \left( q(z),\frac{zq'(z)+(\kappa -1)q(z)}{\kappa },\frac{z^{2}q''(z)+2(\kappa -1)zq'(z)+(\kappa -1)(\kappa
-2)q(z)}{\kappa (\kappa -1)};z\right) =h(z)  \label{2.9}
\end{equation}
has a solution $q$ with $q(0)=0$ and satisfies one of the following
conditions

\begin{enumerate}
\item $q\in \mathcal{Q}_{0}$ and $\phi \in \Phi _{H}[h,q]$

\item $q$ is univalent in $\mathbb{D}$ and $\phi \in \Phi
_{H}[h,q_{\rho }]$ for some $\rho \in (0,1)$

\item $q$ is univalent in $\mathbb{D}$ and there exists $\rho
_{0}\in (0,1)$ such that $\phi \in \Phi _{H}[h_{\rho },q_{\rho }]$ for all $\rho \in (\rho _{0},1).$
\end{enumerate}

If $f\in \mathcal{A}$ satisfies (\ref{2.8}), then
\begin{equation*}
B_{\kappa +1}^{c}f(z)\prec q(z)\ \ (z\in\mathbb{D})
\end{equation*}
and $q$ is the best dominant.
\end{theorem}

\begin{proof}
Following the same arguments as in \cite[Theorem 2.3e]{Mil1}, we
deduce that $q$ is a dominant from Corollary \ref{T2.2} and Theorem \ref{T2.3}.
Since $q$ satisfies (\ref{2.9}), it is also a solution of (\ref{2.8}) and
therefore $q$ will be dominated by all dominants. Hence $q$ is the
best dominant.
\end{proof}

In the particular case when $q(z)=Mz;$ $M>0,$ and in view of Definition \ref{d2.1}, the class of admissible functions $\Phi _{H}[\Omega ,q]$, denoted by $\Phi
_{H}[\Omega ,M]$, is described below.

\begin{definition}
\label{d2.2} Let $\Omega $ be a set in $\mathbb{C}$ and $M>0.$ The class of admissible functions $\Phi _{H}[\Omega ,M]$
consists of those functions $\phi:\mathbb{C}^{3}\times \mathbb{D}\rightarrow\mathbb{C}$ such that
\begin{equation}
\phi \left( Me^{i\theta },\frac{k+(\kappa -1)}{\kappa }Me^{i\theta },\frac{
L+(\kappa -1)\left( 2k+\kappa -2\right) Me^{i\theta }}{\kappa (\kappa -1)}
;z\right) \notin \Omega ,  \label{2.10}
\end{equation}
whenever $z\in\mathbb{D},$ $\kappa \in\mathbb{C}$ $\left( \kappa \neq 0,1\right) $ and $\real(Le^{-i\theta })\geq
\mathcal{(}k-1\mathcal{)}kM$ for all $\theta \in\mathbb{R},\;k\geq 1.$
\end{definition}

\begin{corollary}
\label{C2.2} Let $\phi \in \Phi _{H}[\Omega ,M].$ If $f\in \mathcal{A}$
satisfies the following inclusion relationship
\begin{equation*}
\phi \left( B_{\kappa +1}^{c}f(z),B_{\kappa }^{c}f(z),B_{\kappa
-1}^{c}f(z);z\right) \in \Omega ,
\end{equation*}%
then%
\begin{equation*}
B_{\kappa +1}^{c}f(z)\prec Mz\text{ \ \ }(z\in\mathbb{D}).
\end{equation*}
\end{corollary}

In the special case when $\Omega =\{w:\;\left\vert w\right\vert <M\}=q(\mathbb{D}),$ the class $\Phi _{H}[\Omega ,M]$ is simply denoted by $\Phi
_{H}[M]$. Corollary \ref{C2.2} can now be written in the following form.

\begin{corollary}
\label{C2.3} Let $\phi \in \Phi _{H}[M].$ If $f\in \mathcal{A}$ satisfies
the following inequality
\begin{equation*}
\left\vert \phi \left( B_{\kappa +1}^{c}f(z),B_{\kappa }^{c}f(z),B_{\kappa
-1}^{c}f(z);z\right) \right\vert <M,
\end{equation*}
then
\begin{equation*}
\left\vert B_{\kappa +1}^{c}f(z)\right\vert <M.
\end{equation*}
\end{corollary}

\begin{corollary}
\label{C2.4} If $M>0,$ $\real (\kappa )\geq \frac{1-k}{2},$ where $k\geq 1,$ and $f\in \mathcal{A}$ satisfies the following inequality
\begin{equation*}
\left\vert B_{\kappa }^{c}f(z)\right\vert <M,
\end{equation*}
then
\begin{equation*}
\left\vert B_{\kappa +1}^{c}f(z)\right\vert <M.
\end{equation*}
\end{corollary}

\begin{proof}
This follows from Corollary \ref{C2.3} by taking $\phi (u,v,w;z)=v=\frac{%
k+(\kappa -1)}{\kappa }Me^{i\theta }.$
\end{proof}

Taking into account the above results, we have the following particular
cases. For $f(z)=\frac{z}{1-z}$ in Corollary \ref{C2.4} we have
\begin{equation}
\left\vert \varphi _{\kappa ,c}(z)\right\vert <M\Rightarrow \left\vert
\varphi _{\kappa +1,c}(z)\right\vert <M.  \label{2.111}
\end{equation}%
This result is the generalization of a result given by Prajapat \cite{Pr}.

Also observe that $\varphi _{\frac{3}{2},1,1}(z)=\frac{3\sin \sqrt{z}}{\sqrt{z}}%
-3\cos \sqrt{z},$ $\varphi _{\frac{1}{2},1,1}=\sqrt{z}\sin \sqrt{z}$ and $\varphi
_{-\frac{1}{2},1,1}=z\cos \sqrt{z}$ where $\varphi _{p,b,c}(z)$ is given by (\ref{8}%
). Thus we can obtain some trigonometric inequalities for special cases of
parameters $p,b$ and $c.$ For example from (\ref{2.111}) for all $z\in\mathbb{D}$ and $M>0$ we have
\begin{equation*}
\left\vert z\cos \sqrt{z}\right\vert <M\Rightarrow \left\vert \sqrt{z}\sin
\sqrt{z}\right\vert <M\Rightarrow \left\vert \frac{\sin \sqrt{z}}{\sqrt{z}}%
-\cos \sqrt{z}\right\vert <\frac{M}{3}.
\end{equation*}

\begin{corollary}
\label{C2.5} Let $\kappa \in\mathbb{C}\setminus\{0\}$ and $M>0.$ If $f\in \mathcal{A}$ satisfies the following
inequality
\begin{equation*}
\left\vert B_{\kappa }^{c}f(z)-B_{\kappa +1}^{c}f(z)\right\vert <\frac{M}{\left\vert \kappa \right\vert },
\end{equation*}
then
\begin{equation*}
\left\vert B_{\kappa +1}^{c}f(z)\right\vert <M.
\end{equation*}
\end{corollary}

\begin{proof}
Let $\phi (u,v,w;z)=v+\left( \frac{1}{\kappa }-1\right) u$ and $\Omega =h(\mathbb{D})$ where $h(z)=\frac{Mz}{\kappa },$ $M>0.$ In order to
use Corollary \ref{C2.2}, we need to show that $\phi \in \Phi _{H}[\Omega
,M],$ that is, the admissibility condition (\ref{2.10}) is satisfied. This
follows since
\begin{eqnarray*}
\left\vert \phi \left( Me^{i\theta },\frac{k+(\kappa -1)}{\kappa }
Me^{i\theta },\frac{L+(\kappa -1)\left( 2k+\kappa -2\right) Me^{i\theta }}{%
\kappa (\kappa -1)};z\right) \right\vert=\left\vert \frac{kMe^{i\theta }}{\kappa }\right\vert \geq \frac{M}{\left\vert \kappa \right\vert }
\end{eqnarray*}
whenever $z\in\mathbb{D}$, $\theta \in\mathbb{R},$ $\kappa \in\mathbb{C}$ $\left( \kappa \neq 0,1\right) $ and$\;k\geq 1.$ The result now
follows from Corollary \ref{C2.2}.
\end{proof}

Observe that $\varphi _{\frac{3}{2},1,-1}(z)=3\cosh \sqrt{z}-\frac{3\sinh \sqrt{z}}{\sqrt{z}},$ $\varphi _{\frac{1}{2},1,-1}=\sqrt{z}\sinh \sqrt{z}$ and $\varphi
_{-\frac{1}{2},1,-1}=z\cosh \sqrt{z}$ where $\varphi _{p,b,c}(z)$ is given by (\ref{8}). Moreover, if we take $f(z)=\frac{z}{1-z},$ $b=1,$ $c=-1;$ $p=\frac{1}{2}$ and $p=-\frac{1}{2}$ in Corollary \ref{C2.5}, respectively, we have
\begin{equation*}
\left\vert \frac{(z-3)\sinh \sqrt{z}}{\sqrt{z}}-3\cosh \sqrt{z}\right\vert <
\frac{2M}{3}\Rightarrow \left\vert \frac{\sinh \sqrt{z}}{\sqrt{z}}-\cosh
\sqrt{z}\right\vert <\frac{M}{3}
\end{equation*}
and
\begin{equation*}
\left\vert z\cosh \sqrt{z}-\sqrt{z}\sinh \sqrt{z}\right\vert <2M\Rightarrow
\left\vert \sqrt{z}\sinh \sqrt{z}\right\vert <M.
\end{equation*}

Theorem \ref{T2.4} shows that the result is sharp. The differential equation
\begin{equation*}
zq'(z)=Mz
\end{equation*}
has a univalent solution $q(z)=Mz.$ It follows from Theorem \ref{T2.4} that $q(z)=Mz$ is the best dominant.

\begin{definition}
\label{d2.3} Let $\Omega $ be a set in $\mathbb{C}$ and $q\in \mathcal{Q}_{1}\cap \mathcal{H}_{1\text{ }}.$ The class of
admissible functions $\Phi _{H,1}[\Omega ,q]$ consists of those functions $\phi:\mathbb{C}^{3}\times \mathbb{D}\rightarrow\mathbb{C}$ that satisfy the admissibility condition
\begin{equation*}
\phi (u,v,w;z)\notin \Omega
\end{equation*}
whenever
\begin{equation*}
u=q(\zeta ),\ v=\frac{1}{\kappa -1}\left( \frac{k\zeta q'(\zeta )}{%
q(\zeta )}+\kappa q(\zeta )-1\right) \text{ \ \ }(\kappa \in\mathbb{C},\text{ }\kappa \neq 0,1,2,\text{ }q(\zeta )\neq 0),
\end{equation*}
and
\begin{equation*}
\real\left( \frac{(\kappa -1)v\left[ (\kappa -1)(w-v)+1-w\right] }{
(\kappa -1)v+1-\kappa u}-\left[ 2\kappa u-1-(\kappa -1)v\right] \right) \geq
k\real\left( \frac{\zeta q''(\zeta )}{q'(\zeta )}+1\right)
\end{equation*}
\begin{equation*}
\left( z\in \mathrm{{\mathbb{D}}};\;\zeta \in \partial\mathbb{D}\setminus E(q);\;k\geq 1\right) .
\end{equation*}
\end{definition}

\begin{theorem}
\label{T2.5} Let $\phi \in \Phi _{H,1}[\Omega ,q].$ If $f\in \mathcal{A}$
satisfies the following inclusion relationship
\begin{equation}
\left\{ \phi \left( \frac{B_{\kappa }^{c}f(z)}{B_{\kappa +1}^{c}f(z)},\frac{
B_{\kappa -1}^{c}f(z)}{B_{\kappa }^{c}f(z)},\frac{B_{\kappa -2}^{c}f(z)}{
B_{\kappa -1}^{c}f(z)};z\right) :z\in\mathbb{D}\right\} \subset
\Omega ,  \label{2.11}
\end{equation}
then
\begin{equation*}
\frac{B_{\kappa }^{c}f(z)}{B_{\kappa +1}^{c}f(z)}\prec q(z)\text{ \ \ }(z\in
\mathrm{{\mathbb{D}}}).
\end{equation*}
\end{theorem}

\begin{proof}
Let us consider the analytic function ${p}:\mathbb{D}\to\mathbb{C},$ defined by
\begin{equation}
{p}(z)=\frac{B_{\kappa }^{c}f(z)}{B_{\kappa +1}^{c}f(z)}.
\label{2.12}
\end{equation}
Differentiating both sides of (\ref{2.12}) with respect to $z$ and using (\ref{12}),
we have
\begin{equation}
\frac{B_{\kappa -1}^{c}f(z)}{B_{\kappa }^{c}f(z)}=\frac{1}{\kappa -1}\left(
\frac{z{p}'(z)}{{p}(z)}+\kappa {p}
(z)-1\right) .  \label{2.13}
\end{equation}
Further computations show that
\begin{equation}
\frac{B_{\kappa -2}^{c}f(z)}{B_{\kappa -1}^{c}f(z)}=\frac{1}{\kappa -2}
\left( \frac{z{p}'(z)}{{p}(z)}+\kappa {p}
(z)-2+\frac{\frac{z{p}'(z)}{{p}(z)}+\frac{z^{2}{p}''(z)}{{p}(z)}-\left( \frac{z{p}'(z)}{{p}(z)}\right)^{2}+\kappa z{p}'(z)}{\frac{z{p}'(z)}{{p}(z)}+\kappa {p}(z)-1}\right) .  \label{2.14}
\end{equation}
Now we define the transformation $\psi:\mathbb{C}^{3}\times\mathbb{D}\to\mathbb{C}$ by $\psi (r,s,t;z)=\phi (u,v,w;z)$ where
\begin{equation}
u=r,\text{ \ }v=\frac{1}{\kappa -1}\left( \frac{s}{r}+\kappa r-1\right)
\text{ \ and \ }w=\frac{1}{\kappa -2}\left( \frac{s}{r}+\kappa r-2+\frac{
\frac{s}{r}+\frac{t}{r}-\left( \frac{s}{r}\right) ^{2}+\kappa s}{\frac{s}{r}
+\kappa r-1}\right) .  \label{2.15}
\end{equation}
Using equations (\ref{2.12})-(\ref{2.14}), we get
\begin{equation}
\psi ({p}(z),z{p}'(z),z^{2}{p}''(z);z)=\phi \left( \frac{B_{\kappa }^{c}f(z)}{B_{\kappa +1}^{c}f(z)},
\frac{B_{\kappa -1}^{c}f(z)}{B_{\kappa }^{c}f(z)},\frac{B_{\kappa -2}^{c}f(z)
}{B_{\kappa -1}^{c}f(z)};z\right) .  \label{2.17}
\end{equation}
Hence (\ref{2.11}) implies
\begin{equation*}
\psi ({p}(z),z{p}'(z),z^{2}{p}''(z);z)\in \Omega .
\end{equation*}
The proof is completed if it can be shown that the admissibility condition
for $\phi \in \Phi _{H,1}[\Omega ,q]$ is equivalent to the admissibility
condition for $\psi $ as given in Definition \ref{d1}. Note that
\begin{equation*}
\frac{t}{s}+1=\frac{(\kappa -1)v\left[ (\kappa -1)(w-v)+1-w\right] }{(\kappa
-1)v+1-\kappa u}-\left[ 2\kappa u-1-(\kappa -1)v\right] ,
\end{equation*}
and hence $\psi \in \Psi \lbrack \Omega ,q].$ By Lemma \ref{l1}, we have
\begin{equation*}
{p}(z)\prec q(z)\text{ \ or \ }\frac{B_{\kappa }^{c}f(z)}{B_{\kappa
+1}^{c}f(z)}\prec q(z)
\end{equation*}
which completes the proof of Theorem \ref{T2.5}.
\end{proof}

In the case when $\Omega \neq\mathbb{C}$ is a simply connected domain with $\Omega =h(\mathbb{D})$ for some
conformal mapping $h$ of $\mathbb{D}$ onto $\Omega .$ In this case the
class $\Phi _{H,1}[h(\mathbb{D}),q]$ is written as $\Phi _{H,1}[h,q].$

The following result is an immediate consequence of Theorem \ref{T2.5}.

\begin{corollary}
\label{T2.6} Let $\phi \in \Phi _{H,1}[h,q].$ If $f\in \mathcal{A}$ satisfies
\begin{equation}
\phi \left( \frac{B_{\kappa }^{c}f(z)}{B_{\kappa +1}^{c}f(z)},\frac{
B_{\kappa -1}^{c}f(z)}{B_{\kappa }^{c}f(z)},\frac{B_{\kappa -2}^{c}f(z)}{
B_{\kappa -1}^{c}f(z)};z\right) \prec h(z),  \label{2.18}
\end{equation}
then
\begin{equation*}
\frac{B_{\kappa }^{c}f(z)}{B_{\kappa +1}^{c}f(z)}\prec q(z)\text{ \ \ }(z\in\mathbb{D}).
\end{equation*}
\end{corollary}

In the particular case when $q(z)=1+Mz;$ $M>0,$ the class of admissible functions
$\Phi _{H,1}[\Omega ,q]$ is simply denoted by $\Phi _{H,1}[\Omega ,M]$.

\begin{definition}
\label{d2.4} Let $\Omega $ be a set in $\mathbb{C}$, and $M>0.$ The class of admissible functions $\Phi _{H,1}[\Omega ,M]$
consists of those functions $\phi:\mathbb{C}^{3}\times \mathbb{D}\rightarrow\mathbb{C}$ such that
\begin{eqnarray}
&&\phi \left( 1+Me^{i\theta },1+\frac{k+\kappa (1+Me^{i\theta })}{(\kappa
-1)(1+Me^{i\theta })}Me^{i\theta },1+\frac{k+\kappa (1+Me^{i\theta })}{%
(\kappa -2)(1+Me^{i\theta })}Me^{i\theta }\right.  \label{2.19} \\
&&\left. +\frac{(M+e^{-i\theta })\left[ Le^{-i\theta }+(1+\kappa )k M+\kappa
kM^{2}e^{i\theta }\right] -k^{2}M^{2}}{(\kappa -2)(M+e^{-i\theta })\left[
(\kappa -1)e^{-i\theta }+\kappa M^{2}e^{i\theta }+(1+k+2\kappa )M\right] }%
\right) \notin \Omega  \notag
\end{eqnarray}
whenever $z\in\mathbb{D},$ $\kappa \in\mathbb{C}$ $(\kappa \neq 0,1,2)$ and $\real(Le^{-i\theta })\geq \mathcal{(}%
k-1\mathcal{)}kM$ for all $\theta \in\mathbb{R},\;k\geq 1.$
\end{definition}

\begin{corollary}
\label{C2.6} Let $\phi \in \Phi _{H,1}[\Omega ,M].$ If $f\in \mathcal{A}$
satisfies the following inclusion relationship
\begin{equation*}
\phi \left( \frac{B_{\kappa }^{c}f(z)}{B_{\kappa +1}^{c}f(z)},\frac{%
B_{\kappa -1}^{c}f(z)}{B_{\kappa }^{c}f(z)},\frac{B_{\kappa -2}^{c}f(z)}{%
B_{\kappa -1}^{c}f(z)};z\right) \in \Omega ,
\end{equation*}
then%
\begin{equation*}
\frac{B_{\kappa }^{c}f(z)}{B_{\kappa +1}^{c}f(z)}-1\prec Mz\text{ \ \ }(z\in\mathbb{D}).
\end{equation*}
\end{corollary}

In the special case $\Omega =\{w:\;\left\vert w-1\right\vert <M\}=q(\mathbb{D}),$ the class $\Phi _{H,1}[\Omega ,M]$ is simply denoted by $\Phi _{H,1}[M]$, and Corollary \ref{C2.6} takes the following form.

\begin{corollary}
\label{C2.7} Let $\phi \in \Phi _{H,1}[M].$ If $f\in \mathcal{A}$ satisfies
the following inequality
\begin{equation*}
\left\vert \phi \left( \frac{B_{\kappa }^{c}f(z)}{B_{\kappa +1}^{c}f(z)},%
\frac{B_{\kappa -1}^{c}f(z)}{B_{\kappa }^{c}f(z)},\frac{B_{\kappa -2}^{c}f(z)%
}{B_{\kappa -1}^{c}f(z)};z\right) -1\right\vert <M,
\end{equation*}
then for all $z\in\mathbb{D}$ we have
\begin{equation*}
\left\vert \frac{B_{\kappa }^{c}f(z)}{B_{\kappa +1}^{c}f(z)}-1\right\vert <M.
\end{equation*}
\end{corollary}

\begin{corollary}
\label{C2.8} Let $\kappa \in\mathbb{C}$ $(\kappa \neq 1)$ and $M>0.$ If $f\in \mathcal{A}$ satisfies the following
inequality
\begin{equation*}
\left\vert \frac{B_{\kappa -1}^{c}f(z)}{B_{\kappa }^{c}f(z)}-\frac{B_{\kappa
}^{c}f(z)}{B_{\kappa +1}^{c}f(z)}\right\vert <\frac{M^{2}}{\left\vert \kappa
-1\right\vert (1+M)},
\end{equation*}
then for all $z\in\mathbb{D}$ we have
\begin{equation*}
\left\vert \frac{B_{\kappa }^{c}f(z)}{B_{\kappa +1}^{c}f(z)}-1\right\vert <M.
\end{equation*}
\end{corollary}

\begin{proof}
This follows from Corollary \ref{C2.6} by taking $\phi (u,v,w;z)=v-\frac{\kappa }{\kappa -1}(u-1)-1$ and $\Omega =h(\mathbb{D}),$ where
$$h(z)=\frac{M}{\left\vert \kappa -1\right\vert (1+M)}z,$$ and $M>0.$ In order to
use Corollary \ref{C2.6}, we need to show that $\phi \in \Phi _{H,1}[\Omega
,M],$ that is, the admissibility condition (\ref{2.19}) is satisfied. This
follows since
\begin{eqnarray*}
\left\vert \phi \left( u,v,w;z\right) \right\vert &=&\left\vert 1+\frac{%
k+\kappa (1+Me^{i\theta })}{(\kappa -1)(1+Me^{i\theta })}Me^{i\theta }-\frac{%
\kappa }{\kappa -1}\left( 1+Me^{i\theta }-1\right) -1\right\vert \\
&=&\frac{M}{\left\vert \kappa -1\right\vert }\left\vert \frac{k}{%
1+Me^{i\theta }}\right\vert \geq \frac{M}{\left\vert \kappa -1\right\vert
(1+M)}
\end{eqnarray*}
when $z\in\mathbb{D}$, $\theta \in\mathbb{R},$ $\kappa \in\mathbb{C}$ $(\kappa \neq 1)$ and$\;k\geq 1.$ The result now follows from
Corollary \ref{C2.6}.
\end{proof}

\begin{definition}
\label{d2.5} Let $\Omega $ be a set in $\mathbb{C}$ and $q\in \mathcal{Q}_{1}\cap \mathcal{H}_{1\text{ }}.$ The class of
admissible functions $\Phi _{H,2}[\Omega ,q]$ consists of those functions $\phi:\mathbb{C}^{3}\times \mathbb{D}\rightarrow\mathbb{C}$ that satisfy the admissibility condition
\begin{equation*}
\phi (u,v,w;z)\notin \Omega
\end{equation*}
where
\begin{equation*}
u=q(\zeta ),\ v=\frac{k\zeta q^{\prime }(\zeta )+\kappa q(\zeta )}{\kappa }
\text{ \ \ }\left( \kappa \in\mathbb{C},\text{ }\kappa \neq 0,1\right)
\end{equation*}
and
\begin{equation*}
\real\left( \frac{(\kappa -1)(w-u)}{v-u}+(1-2\kappa )\right) \geq
k\real\left( \frac{\zeta q''(\zeta )}{q'(\zeta )}+1\right)
\end{equation*}
\begin{equation*}
\left( z\in\mathbb{D};\;\zeta \in \partial\mathbb{D}\setminus E(q);\;k\geq 1\right) .
\end{equation*}
\end{definition}

\begin{theorem}
\label{T2.7} Let $\phi \in \Phi _{H,2}[\Omega ,q].$ If $f\in \mathcal{A}$
satisfies the following inclusion relationship
\begin{equation}
\left\{ \phi \left( \frac{B_{\kappa +1}^{c}f(z)}{z},\frac{B_{\kappa }^{c}f(z)%
}{z},\frac{B_{\kappa -1}^{c}f(z)}{z};z\right) :z\in \mathbb{D}
\right\} \subset \Omega ,  \label{4.1}
\end{equation}%
then%
\begin{equation*}
\frac{B_{\kappa +1}^{c}f(z)}{z}\prec q(z)\text{ \ \ }(z\in\mathbb{D}).
\end{equation*}
\end{theorem}

\begin{proof}
Let us define the analytic function ${p}$ in $\mathbb{D}$ by
\begin{equation}
p(z)=\frac{B_{\kappa +1}^{c}f(z)}{z}.  \label{4.2}
\end{equation}
By making use of (\ref{12}) and (\ref{4.2}), we get
\begin{equation}
\frac{B_{\kappa }^{c}f(z)}{z}=\frac{z{p}'(z)+\kappa{p}(z)}{\kappa }.  \label{4.3}
\end{equation}
Further computations show that
\begin{equation}
\frac{B_{\kappa -1}^{c}f(z)}{z}=\frac{z^2{p}''(z)+2\kappa z{p}'(z)+\kappa (\kappa -1){p}(z)}{\kappa (\kappa -1)}.  \label{4.4}
\end{equation}
Define the transformations from $\mathbb{C}^{3}$ to $\mathbb{C}$ by
\begin{equation}
u=r,\text{ \ }v=\frac{s+\kappa r}{\kappa }\text{ \ and \ }w=\frac{t+2\kappa
s+\kappa (\kappa -1)r}{\kappa (\kappa -1)}.  \label{4.5}
\end{equation}
Let
\begin{equation}
\psi (r,s,t;z)=\phi (u,v,w;z)=\phi \left( r,\frac{s+\kappa r}{\kappa },\frac{
t+2\kappa s+\kappa (\kappa -1)r}{\kappa (\kappa -1)};z\right) .  \label{4.6}
\end{equation}
The proof shall make use of Lemma \ref{l1}. Using equations (\ref{4.2}), (\ref{4.3}) and (\ref{4.4}), from (\ref{4.6}) we obtain
\begin{equation}
\psi ({p}(z),z{p}'(z),z^{2}{p}''(z);z)=\phi \left( \frac{B_{\kappa +1}^{c}f(z)}{z},\frac{B_{\kappa
}^{c}f(z)}{z},\frac{B_{\kappa -1}^{c}f(z)}{z};z\right) .  \label{4.7}
\end{equation}
Hence (\ref{4.1}) becomes
\begin{equation*}
\psi ({p}(z),z{p}'(z),z^{2}{p}''(z);z)\in \Omega .
\end{equation*}
The proof is completed if it can be shown that the admissibility condition
for $\phi \in \Phi _{H,2}[\Omega ,q]$ is equivalent to the admissibility
condition for as given in Definition \ref{d1}. Note that
\begin{equation*}
\frac{t}{s}+1=\frac{(\kappa -1)(w-u)}{v-u}+(1-2\kappa ),
\end{equation*}%
and hence $\psi \in \Psi \lbrack \Omega ,q].$ By Lemma \ref{l1}, we have ${p}(z)\prec q(z)$ or
\begin{equation*}
\frac{B_{\kappa +1}^{c}f(z)}{z}\prec q(z).
\end{equation*}
\end{proof}

If\ $\Omega \neq\mathbb{C}$ is a simply connected domain, then $\Omega =h(\mathbb{D})$ for some
conformal mapping $h$ of $\mathbb{D}$ onto $\Omega .$ In this case the
class $\Phi _{H,2}[h(\mathbb{D}),q]$ is written as $\Phi _{H,2}[h,q].$

The following result is an immediate consequence of Theorem \ref{T2.7}.

\begin{corollary}
\label{T2.8} Let $\phi \in \Phi _{H,2}[h,q].$ If $f\in \mathcal{A}$ satisfies
\begin{equation}
\phi \left( \frac{B_{\kappa +1}^{c}f(z)}{z},\frac{B_{\kappa }^{c}f(z)}{z},
\frac{B_{\kappa -1}^{c}f(z)}{z};z\right) \prec h(z),  \label{4.8}
\end{equation}
then
\begin{equation*}
\frac{B_{\kappa +1}^{c}f(z)}{z}\prec q(z)\text{ \ \ }(z\in \mathbb{D}).
\end{equation*}
\end{corollary}

In the particular case when $q(z)=1+Mz;$ $M>0,$ the class of admissible functions
$\Phi _{H,2}[\Omega ,q]$, denoted by $\Phi _{H,2}[\Omega ,M]$.

\begin{definition}
\label{d2.6} Let $\Omega $ be a set in $\mathbb{C}$ and $M>0.$ The class of admissible functions $\Phi _{H,2}[\Omega ,M]$
consists of those functions $\phi:\mathbb{C}^{3}\times \mathbb{D}\rightarrow\mathbb{C}$ such that
\begin{equation}
\phi \left( 1+Me^{i\theta },1+\frac{k+\kappa }{\kappa }Me^{i\theta },1+\frac{%
L+\kappa (2k+\kappa -1)Me^{i\theta }}{\kappa (\kappa -1)};z\right) \notin
\Omega  \label{4.9}
\end{equation}
whenever $z\in \mathrm{{\mathbb{D}}},$ $\kappa \in\mathbb{C}$ $(\kappa \neq 0,1)$ and $\real (Le^{-i\theta })\geq \mathcal{(}k-1%
\mathcal{)}kM$ for all $\theta \in\mathbb{R},\;k\geq 1.$
\end{definition}

\begin{corollary}
\label{C2.9} Let $\phi \in \Phi _{H,2}[\Omega ,M].$ If $f\in \mathcal{A}$
satisfies the following inclusion relationship
\begin{equation*}
\phi \left( \frac{B_{\kappa +1}^{c}f(z)}{z},\frac{B_{\kappa }^{c}f(z)}{z},%
\frac{B_{\kappa -1}^{c}f(z)}{z};z\right) \in \Omega ,
\end{equation*}
then
\begin{equation*}
\frac{B_{\kappa +1}^{c}f(z)}{z}-1\prec Mz\text{ \ \ }(z\in \mathbb{D}).
\end{equation*}
\end{corollary}

In the special case when $\Omega =\{w:\;\left\vert w-1\right\vert <M\}=q(\mathbb{D}),$ the class $\Phi _{H,2}[\Omega ,M]$ is simply denoted by $\Phi _{H,2}[M]$, and Corollary \ref{C2.9} takes the following form.

\begin{corollary}
\label{C2.10} Let $\phi \in \Phi _{H,2}[M].$ If $f\in \mathcal{A}$ satisfies
the following inequality
\begin{equation*}
\left\vert \phi \left( \frac{B_{\kappa +1}^{c}f(z)}{z},\frac{B_{\kappa
}^{c}f(z)}{z},\frac{B_{\kappa -1}^{c}f(z)}{z};z\right) -1\right\vert <M,
\end{equation*}
then for all $z\in\mathbb{D}$ we have
\begin{equation*}
\left\vert \frac{B_{\kappa +1}^{c}f(z)}{z}-1\right\vert <M.
\end{equation*}
\end{corollary}

\begin{corollary}
\label{C2.11} Let $\kappa \neq 0$ and $M>0.$ If $f\in \mathcal{A}$ satisfies
the following inequality
\begin{equation*}
\left\vert \frac{B_{\kappa }^{c}f(z)}{z}-\frac{B_{\kappa +1}^{c}f(z)}{z}%
\right\vert <\frac{M}{\left\vert \kappa \right\vert },
\end{equation*}%
then for all $z\in\mathbb{D}$ we have
\begin{equation*}
\left\vert \frac{B_{\kappa +1}^{c}f(z)}{z}-1\right\vert <M.
\end{equation*}
\end{corollary}

\begin{proof}
This follows from Corollary \ref{C2.9} by taking $\phi (u,v,w;z)=v-u.$
\end{proof}

\begin{corollary}
\label{C2.12} Let $\real (\kappa )\geq -\frac{k}{2},$ $\kappa \neq 0,$ $k\geq1$
and $M>0.$ If $f\in \mathcal{A}$ satisfy the following inequality
\begin{equation*}
\left\vert \frac{B_{\kappa }^{c}f(z)}{z}-1\right\vert <M,
\end{equation*}
then for all $z\in\mathbb{D}$ we have
\begin{equation*}
\left\vert \frac{B_{\kappa +1}^{c}f(z)}{z}-1\right\vert <M.
\end{equation*}
\end{corollary}

\begin{proof}
This follows from Corollary \ref{C2.9} by taking $\phi (u,v,w;z)=v-1.$
\end{proof}

For $f(z)=z/(1-z)$ in Corollary \ref{C2.12} we have
\begin{equation}
\left\vert \frac{\varphi _{\kappa ,c}(z)}{z}-1\right\vert <M\Rightarrow
\left\vert \frac{\varphi _{\kappa +1,c}(z)}{z}-1\right\vert <M.  \label{4.10}
\end{equation}
This result is somewhat related to an open problem given by Andr\'as and
Baricz \cite{An}, which in terms of $\varphi_{\kappa,c}$ can be rewritten as follows: is it true that
if $p>-1$ increase, then the image region $\varphi_{p,1,1}(\mathbb{D})$ decrease, that is, if $p>q>-1,$ then $\varphi_{p,1,1}(\mathbb{D})\subset\varphi_{q,1,1}(\mathbb{D})$?

In particular, if we take $b=c=1;$ $p=-\frac{1}{2}$ and $p=\frac{1}{2},$ in (\ref{4.10}), then the following chain of implications is true
\begin{equation*}
\left\vert \cos \sqrt{z}-1\right\vert <M\Rightarrow \left\vert \frac{\sin
\sqrt{z}}{\sqrt{z}}-1\right\vert <M\Rightarrow \left\vert \frac{\sin \sqrt{z}%
}{z\sqrt{z}}-\frac{\cos \sqrt{z}}{z}-\frac{1}{3}\right\vert <\frac{M}{3},
\end{equation*}
where $\kappa$ is as in Corollary \ref{C2.12} and $M>0.$

\section{\textbf{Superordination and sandwich-type results involving the }$%
B_{\protect\kappa }^{c}-$\textbf{operator}}

In this section we obtain some differential superordination results for functions
associated with the $B_{\kappa }^{c}-$operator defined by (\ref{12}). As in the previous section, the
parameter $c$ is an arbitrary complex number, while for the parameter $\kappa$
we will have some assumptions in some cases. We consider first a class of admissible functions, which is given in the following
definition.

\begin{definition}
\label{d3.1} Let $\Omega $ be a set in $\mathbb{C}$, $q\in \mathcal{H}_{0\text{ }}$with $zq'(z)\neq 0.$ The class
of admissible functions $\Phi _{H}'[\Omega ,q]$ consists of those
functions $\phi:\mathbb{C}^{3}\times \overline{\mathbb{D}}\rightarrow\mathbb{C}$ that satisfy the admissibility condition
\begin{equation*}
\phi (u,v,w;\varsigma )\notin \Omega
\end{equation*}%
whenever
\begin{equation*}
u=q(z),\ v=\frac{zq'(z)+m(\kappa -1)q(z)}{m\kappa },\text{ \ \ }
\left( \kappa \in\mathbb{C},\text{ }\kappa \neq 0,1\right)
\end{equation*}
and
\begin{equation*}
\real\left( \frac{\kappa (\kappa -1)w-(\kappa -2)(\kappa -1)u}{%
v\kappa -(\kappa -1)u}-(2\kappa -3)\right) \leq \frac{1}{m}\real
\left( \frac{zq''(z)}{q'(z)}+1\right)
\end{equation*}
\begin{equation*}
\left( z\in\mathbb{D};\;\varsigma \in \partial\mathbb{D};\;m\geq 1\right) .
\end{equation*}
\end{definition}

\begin{theorem}
\label{T3.1} Let $\phi \in \Phi _{H}^{\prime }[\Omega ,q].$ If $f\in
\mathcal{A},$ $B_{\kappa +1}^{c}f\in \mathcal{Q}_{0}$ and%
\begin{equation*}
\phi \left( B_{\kappa +1}^{c}f(z),B_{\kappa }^{c}f(z),B_{\kappa
-1}^{c}f(z);z\right)
\end{equation*}%
is univalent in $\mathrm{{\mathbb{D}}},$ then
\begin{equation}
\Omega \subset \left\{ \phi \left( B_{\kappa +1}^{c}f(z),B_{\kappa
}^{c}f(z),B_{\kappa -1}^{c}f(z);z\right) :z\in \mathrm{{\mathbb{D}}}\right\}
\label{3.1}
\end{equation}%
implies%
\begin{equation*}
q(z)\prec B_{\kappa +1}^{c}f(z)\text{ \ \ }(z\in \mathrm{{\mathbb{D}}}).
\end{equation*}
\end{theorem}

\begin{proof}
Let ${p}(z)$ be defined by (\ref{2.2}) and $\psi $ by (\ref{2.6}).
Since $\phi \in \Phi _{H}'[\Omega ,q],$ (\ref{2.7}) and (\ref{3.1})
yield%
\begin{equation*}
\Omega \subset \left\{ \psi ({p}(z),z{p}'(z),z^{2}{p}''(z);z):z\in\mathbb{D}\right\} .
\end{equation*}%
From (\ref{2.6}), we see that the admissibility condition for $\phi \in \Phi
_{H}'[\Omega ,q]$ is equivalent to the admissibility condition for $\psi $ as given in Definition \ref{d2}. Hence $\psi \in \Psi'[\Omega ,q]$, and by Lemma \ref{l2}, we have
\begin{equation*}
q(z)\prec {p}(z)\text{\ \ \ or \ \ }q(z)\prec B_{\kappa +1}^{c}f(z)
\text{ \ \ }(z\in \mathbb{D})
\end{equation*}
which completes the proof of Theorem \ref{T3.1}.
\end{proof}

If \ $\Omega \neq\mathbb{C}$ is a simply-connected domain, then $\Omega =h(\mathbb{D})$ for some
conformal mapping $h$ of $\mathbb{D}$ onto $\Omega .$ In this case, the
class $\Phi _{H}'[h(\mathbb{D}),q]$ is written as $\Phi
_{H}'[h,q].$

Proceeding similarly as in the previous section, the following result is an
immediate consequence of Theorem \ref{T3.1}.

\begin{corollary}
\label{T3.2} Let $q\in \mathcal{H}_{0},$ $h$ be analytic in $\mathbb{D}$ and $\phi\in\Phi _{H}'[h,q].$ If $f\in \mathcal{A},$ $B_{\kappa +1}^{c}f\in \mathcal{Q}_{0}$ and
\begin{equation*}
\phi \left( B_{\kappa +1}^{c}f(z),B_{\kappa }^{c}f(z),B_{\kappa
-1}^{c}f(z);z\right)
\end{equation*}
is univalent in $\mathbb{D},$ then
\begin{equation}
h(z)\prec \phi \left( B_{\kappa +1}^{c}f(z),B_{\kappa }^{c}f(z),B_{\kappa
-1}^{c}f(z);z\right)  \label{3.2}
\end{equation}
implies
\begin{equation*}
q(z)\prec B_{\kappa +1}^{c}f(z)\text{ \ \ }(z\in\mathbb{D}).
\end{equation*}
\end{corollary}

Theorem \ref{T3.1} and Corollary \ref{T3.2} can only be used to obtain subordinants
of differential superordination of the form (\ref{3.1}) or (\ref{3.2}). The
following theorem proves the existence of the best subordinant of (\ref{3.2}) for an appropriate $\phi $.

\begin{theorem}
\label{T3.3} Let $h$ be univalent in $\mathbb{D}$ and $\phi:\mathbb{C}^{3}\times \overline{\mathbb{D}}\rightarrow\mathbb{C}.$ Suppose that the differential equation
\begin{equation*}
\phi \left( q(z),\frac{zq'(z)+(\kappa -1)q(z)}{\kappa },\frac{z^{2}q''(z)+2(\kappa -1)zq'(z)+(\kappa -1)(\kappa
-2)q(z)}{\kappa (\kappa -1)};z\right) =h(z)
\end{equation*}
has a solution $q\in \mathcal{Q}_{0}.$ If $\phi \in \Phi _{H}'[h,q],$ $f\in \mathcal{A},$ $B_{\kappa +1}^{c}f\in \mathcal{Q}_{0}$ and
\begin{equation*}
\phi \left( B_{\kappa +1}^{c}f(z),B_{\kappa }^{c}f(z),B_{\kappa
-1}^{c}f(z);z\right)
\end{equation*}
is univalent in $\mathbb{D},$ then
\begin{equation*}
h(z)\prec \phi \left( B_{\kappa +1}^{c}f(z),B_{\kappa }^{c}f(z),B_{\kappa
-1}^{c}f(z);z\right)
\end{equation*}
which implies
\begin{equation*}
q(z)\prec B_{\kappa +1}^{c}f(z)\text{ \ \ }(z\in \mathrm{{\mathbb{D}}})
\end{equation*}%
and $q$ is best subordinant.
\end{theorem}

\begin{proof}
The proof is similar to the proof of Theorem \ref{T2.4} and is therefore we
omit the details.
\end{proof}

Combining Theorem \ref{T2.2} and Corollary \ref{T3.2}, we obtain the following sandwich-type
theorem.

\begin{corollary}
\label{C3.1} Let $h_{1}$ and $q_{1}$ be analytic functions in $\mathbb{D},$ $h_{2}$ be a univalent function in $\mathbb{D},$ $q_{2}\in \mathcal{Q}_{0}$ with $q_{1}(0)=q_{2}(0)=1$ and $\phi \in
\Phi _{H}[h_{2},q_{2}]\cap \Phi _{H}'[h_{1},q_{1}].$ If $f\in
\mathcal{A},$ $B_{\kappa +1}^{c}f\in \mathcal{Q}_{0}\cap \mathcal{H}_{0\text{
}}$ and
\begin{equation*}
\phi \left( B_{\kappa +1}^{c}f(z),B_{\kappa }^{c}f(z),B_{\kappa
-1}^{c}f(z);z\right)
\end{equation*}
is univalent in $\mathrm{{\mathbb{D}}},$ then
\begin{equation*}
h_{1}(z)\prec \phi \left( B_{\kappa +1}^{c}f(z),B_{\kappa
}^{c}f(z),B_{\kappa -1}^{c}f(z);z\right) \prec h_{2}(z)
\end{equation*}
which implies
\begin{equation*}
q_{1}(z)\prec B_{\kappa +1}^{c}f(z)\prec q_{2}(z)\text{ \ \ }(z\in\mathbb{D}).
\end{equation*}
\end{corollary}

\begin{definition}
\label{d3.2} Let $\Omega $ be a set in $\mathbb{C}$, $q\in \mathcal{H}_{1\text{ }}$ with $q(z)\neq 0,$ $zq'(z)\neq
0.$ The class of admissible functions $\Phi _{H,1}'[\Omega ,q]$
consists of those functions $\phi:\mathbb{C}^{3}\times \overline{\mathbb{D}}\rightarrow\mathbb{C}$ that satisfy the admissibility condition
\begin{equation*}
\phi (u,v,w;\varsigma )\in \Omega
\end{equation*}
whenever
\begin{equation*}
u=q(z),\ v=\frac{1}{\kappa -1}\left( \frac{zq'(z)}{mq(z)}+\kappa
q(z)-1\right) \text{ \ \ }(\kappa \in\mathbb{C},\text{ }\kappa \neq 0,1,2)
\end{equation*}
and
\begin{equation*}
\real\left( \frac{(\kappa -1)v\left[ (\kappa -1)(w-v)+1-w\right] }{%
(\kappa -1)v+1-\kappa u}-\left[ 2\kappa u-1-(\kappa -1)v\right] \right)
\leq \frac{1}{m}\real\left( \frac{zq''(z)}{q'(z)}+1\right)
\end{equation*}
\begin{equation*}
\left( z\in\mathbb{D};\;\varsigma \in \partial\mathbb{D};\;m\geq 1\right) .
\end{equation*}
\end{definition}

\begin{theorem}
\label{T3.4} Let $\phi \in \Phi _{H,1}'[\Omega ,q].$ If $f\in
\mathcal{A},$ ${B_{\kappa }^{c}f}/{B_{\kappa +1}^{c}f}\in \mathcal{Q}_{1}$ and
\begin{equation*}
\phi \left( \frac{B_{\kappa }^{c}f(z)}{B_{\kappa +1}^{c}f(z)},\frac{
B_{\kappa -1}^{c}f(z)}{B_{\kappa }^{c}f(z)},\frac{B_{\kappa -2}^{c}f(z)}{
B_{\kappa -1}^{c}f(z)};z\right)
\end{equation*}
is univalent in $\mathbb{D},$ then
\begin{equation}
\Omega \subset \left\{ \phi \left( \frac{B_{\kappa }^{c}f(z)}{B_{\kappa
+1}^{c}f(z)},\frac{B_{\kappa -1}^{c}f(z)}{B_{\kappa }^{c}f(z)},\frac{
B_{\kappa -2}^{c}f(z)}{B_{\kappa -1}^{c}f(z)};z\right) :z\in\mathbb{D}\right\}  \label{3.3}
\end{equation}
which implies
\begin{equation*}
q(z)\prec \frac{B_{\kappa }^{c}f(z)}{B_{\kappa +1}^{c}f(z)}\text{ \ \ }(z\in
\mathrm{{\mathbb{D}}}).
\end{equation*}
\end{theorem}

\begin{proof}
Let ${p}$ be defined by (\ref{2.12}) and $\psi $ by (\ref{2.17}). Since $\phi \in \Phi _{H,1}'[\Omega ,q],$ it follows from (\ref{2.17}) and (\ref{3.3}) that
\begin{equation*}
\Omega \subset \left\{ \psi ({p}(z),z{p}'(z),z^{2}{p}''(z);z):z\in \mathrm{{\mathbb{D}}}\right\} .
\end{equation*}
From (\ref{2.17}), we see that the admissibility condition for $\phi \in
\Phi _{H,1}'[\Omega ,q]$ is equivalent to the admissibility
condition for $\psi $ as given in Definition \ref{d2}. Hence $\psi \in \Psi'[\Omega ,q]$, and by Lemma \ref{l2}, we have
\begin{equation*}
q(z)\prec {p}(z)\text{\ \ \ or \ \ }q(z)\prec \frac{B_{\kappa
}^{c}f(z)}{B_{\kappa +1}^{c}f(z)}\text{ \ \ }(z\in\mathbb{D})
\end{equation*}
which completes the proof of Theorem \ref{T3.4}.
\end{proof}

If $\Omega \neq\mathbb{C}$ is a simply-connected domain, then $\Omega =h(\mathbb{D})$ for some
conformal mapping $h$ of $\mathbb{D}$ onto $\Omega .$ In this case, the
class $\Phi _{H,1}'[h(\mathbb{D}),q]$ is written as $\Phi
_{H,1}'[h,q].$ Proceeding similarly in the previous section, the
following result is an immediate consequence of Theorem \ref{T3.4}.

\begin{theorem}
\label{T3.5} Let $q\in \mathcal{H}_{1},$ $h$ be analytic in $\mathbb{D}$ and $\phi \in \Phi _{H,1}'[h,q].$ If $f\in \mathcal{A},$ ${B_{\kappa }^{c}f}/{B_{\kappa +1}^{c}f}\in \mathcal{Q}_{1}$
and
\begin{equation*}
\phi \left( \frac{B_{\kappa }^{c}f(z)}{B_{\kappa +1}^{c}f(z)},\frac{
B_{\kappa -1}^{c}f(z)}{B_{\kappa }^{c}f(z)},\frac{B_{\kappa -2}^{c}f(z)}{
B_{\kappa -1}^{c}f(z)};z\right)
\end{equation*}
is univalent in $\mathbb{D},$ then
\begin{equation}
h(z)\prec \phi \left( \frac{B_{\kappa }^{c}f(z)}{B_{\kappa +1}^{c}f(z)},%
\frac{B_{\kappa -1}^{c}f(z)}{B_{\kappa }^{c}f(z)},\frac{B_{\kappa -2}^{c}f(z)%
}{B_{\kappa -1}^{c}f(z)};z\right)  \label{3.4}
\end{equation}
which implies
\begin{equation*}
q(z)\prec \frac{B_{\kappa }^{c}f(z)}{B_{\kappa +1}^{c}f(z)}\text{ \ \ }(z\in
\mathrm{{\mathbb{D}}}).
\end{equation*}
\end{theorem}

Combining Theorems \ref{T2.6} and \ref{T3.5}, we obtain the following
sandwich-type theorem.

\begin{corollary}
\label{C3.2} Let $h_{1}$ and $q_{1}$ be analytic functions in $\mathbb{D},$ $h_{2}$ be a univalent function in $\mathbb{D},$ $q_{2}\in \mathcal{Q}_{1}$ with $q_{1}(0)=q_{2}(0)=1$ and $\phi \in
\Phi _{H,1}[h_{2},q_{2}]\cap \Phi _{H,1}'[h_{1},q_{1}].$ If $f\in
\mathcal{A},$ ${B_{\kappa }^{c}f}/{B_{\kappa +1}^{c}f}\in \mathcal{%
Q}_{1}\cap \mathcal{H}_{1\text{ }}$ and%
\begin{equation*}
\phi \left( \frac{B_{\kappa }^{c}f(z)}{B_{\kappa +1}^{c}f(z)},\frac{%
B_{\kappa -1}^{c}f(z)}{B_{\kappa }^{c}f(z)},\frac{B_{\kappa -2}^{c}f(z)}{%
B_{\kappa -1}^{c}f(z)};z\right)
\end{equation*}%
is univalent in $\mathbb{D},$ then
\begin{equation*}
h_{1}(z)\prec \phi \left( \frac{B_{\kappa }^{c}f(z)}{B_{\kappa +1}^{c}f(z)},%
\frac{B_{\kappa -1}^{c}f(z)}{B_{\kappa }^{c}f(z)},\frac{B_{\kappa -2}^{c}f(z)%
}{B_{\kappa -1}^{c}f(z)};z\right) \prec h_{2}(z)
\end{equation*}
which implies
\begin{equation*}
q_{1}(z)\prec \frac{B_{\kappa }^{c}f(z)}{B_{\kappa +1}^{c}f(z)}\prec q_{2}(z)%
\text{ \ \ }(z\in \mathbb{D}).
\end{equation*}
\end{corollary}

\begin{definition}
\label{d3.3} Let $\Omega $ be a set in $\mathbb{C}$, $q\in \mathcal{H}_{1\text{ }}$with $q(z)\neq 0,$ $zq'(z)\neq
0.$ The class of admissible functions $\Phi _{H,2}'[\Omega ,q]$
consists of those functions $\phi:\mathbb{C}^{3}\times \overline{\mathbb{D}}\rightarrow\mathbb{C}$ that satisfy the admissibility condition
\begin{equation*}
\phi (u,v,w;\varsigma )\in \Omega
\end{equation*}
whenever
\begin{equation*}
u=q(z),\ v=\frac{\zeta q'(\zeta )+\kappa mq(\zeta )}{\kappa m}\text{
\ \ }(\kappa \in\mathbb{C},\text{ }\kappa \neq 0,1)
\end{equation*}
and
\begin{equation*}
\real\left( \frac{(\kappa -1)(w-u)}{v-u}+(1-2\kappa )\right) \leq
\frac{1}{m}\real\left(\frac{zq''(z)}{q'(z)}+1\right)
\end{equation*}
\begin{equation*}
\left( z\in\mathbb{D};\;\varsigma \in \partial\mathbb{D};\;m\geq 1\right) .
\end{equation*}
\end{definition}

\begin{theorem}
\label{T3.6} Let $\phi \in \Phi _{H,2}'[\Omega ,q].$ If $f\in
\mathcal{A},$ $\frac{B_{\kappa +1}^{c}f(z)}{z}\in \mathcal{Q}_{1}$ and
\begin{equation*}
\phi \left( \frac{B_{\kappa +1}^{c}f(z)}{z},\frac{B_{\kappa }^{c}f(z)}{z},
\frac{B_{\kappa -1}^{c}f(z)}{z};z\right)
\end{equation*}
is univalent in $\mathbb{D},$ then
\begin{equation}
\Omega \subset \left\{ \phi \left( \frac{B_{\kappa +1}^{c}f(z)}{z},\frac{
B_{\kappa }^{c}f(z)}{z},\frac{B_{\kappa -1}^{c}f(z)}{z};z\right) :z\in
\mathrm{{\mathbb{D}}}\right\}  \label{3.5}
\end{equation}
which implies
\begin{equation*}
q(z)\prec \frac{B_{\kappa +1}^{c}f(z)}{z}\text{ \ \ }(z\in \mathbb{D}).
\end{equation*}
\end{theorem}

\begin{proof}
From (\ref{4.7}) and (\ref{3.5}) we obtain that
\begin{equation*}
\Omega \subset \left\{ \psi ({p}(z),z{p}'(z),z^{2}{p}''(z);z):z\in\mathbb{D}\right\} .
\end{equation*}
From (\ref{4.5}), we see that the admissibility condition for $\phi \in \Phi
_{H,2}'[\Omega ,q]$ is equivalent to the admissibility condition
for $\psi $ as given in Definition \ref{d2}. Hence $\psi \in \Psi'[\Omega ,q]$, and by Lemma \ref{l2}, we have
\begin{equation*}
q(z)\prec {p}(z)\text{\ \ \ or \ \ }q(z)\prec \frac{B_{\kappa
+1}^{c}f(z)}{z}\text{ \ \ }(z\in\mathbb{D}).
\end{equation*}
\end{proof}

If \ $\Omega \neq\mathbb{C}$ is a simply-connected domain, then $\Omega =h(\mathbb{D})$ for some
conformal mapping $h$ of $\mathbb{D}$ onto $\Omega .$ In this case, the
class $\Phi _{H,2}'[h(\mathbb{D}),q]$ is written as $\Phi_{H,2}'[h,q].$ Proceeding similarly in the previous section, the
following result is an immediate consequence of Theorem \ref{T3.6}.

\begin{corollary}
\label{T3.7} Let $q\in \mathcal{H}_{1},$ $h$ be analytic in $\mathbb{D}$ and $\phi \in \Phi _{H,2}'[h,q].$ If $f\in \mathcal{A},$ $\frac{B_{\kappa +1}^{c}f(z)}{z}\in \mathcal{Q}_{1}$ and
\begin{equation*}
\phi \left( \frac{B_{\kappa +1}^{c}f(z)}{z},\frac{B_{\kappa }^{c}f(z)}{z},
\frac{B_{\kappa -1}^{c}f(z)}{z};z\right)
\end{equation*}
is univalent in $\mathbb{D},$ then
\begin{equation}
h(z)\prec \phi \left( \frac{B_{\kappa +1}^{c}f(z)}{z},\frac{B_{\kappa
}^{c}f(z)}{z},\frac{B_{\kappa -1}^{c}f(z)}{z};z\right)  \label{3.6}
\end{equation}
which implies
\begin{equation*}
q(z)\prec \frac{B_{\kappa +1}^{c}f(z)}{z}\text{ \ \ }(z\in \mathbb{D}).
\end{equation*}
\end{corollary}

Combining Corollaries \ref{T2.8} and \ref{T3.7}, we obtain the following
sandwich-type theorem.

\begin{corollary}
\label{C3.3} Let $h_{1}$ and $q_{1}$ be analytic functions in $\mathbb{D},$ $h_{2}$ be a univalent function in $\mathbb{D},$ $q_{2}\in \mathcal{Q}_{1}$ with $q_{1}(0)=q_{2}(0)=1$ and $\phi \in\Phi _{H,2}[h_{2},q_{2}]\cap \Phi _{H,2}'[h_{1},q_{1}].$ If $f\in
\mathcal{A},$ $\frac{B_{\kappa +1}^{c}f(z)}{z}\in \mathcal{Q}_{1}\cap
\mathcal{H}_{1\text{ }}$ and
\begin{equation*}
\phi \left( \frac{B_{\kappa +1}^{c}f(z)}{z},\frac{B_{\kappa }^{c}f(z)}{z},%
\frac{B_{\kappa -1}^{c}f(z)}{z};z\right)
\end{equation*}
is univalent in $\mathbb{D},$ then
\begin{equation*}
h_{1}(z)\prec \phi \left( \frac{B_{\kappa +1}^{c}f(z)}{z},\frac{B_{\kappa
}^{c}f(z)}{z},\frac{B_{\kappa -1}^{c}f(z)}{z};z\right) \prec h_{2}(z)
\end{equation*}
which implies
\begin{equation*}
q_{1}(z)\prec \frac{B_{\kappa +1}^{c}f(z)}{z}\prec q_{2}(z)\text{ \ \ }(z\in\mathbb{D}).
\end{equation*}
\end{corollary}

We note that in particular the above main results reduce to results for the operators $\mathcal{J}_{p}f$, $\mathcal{I}_{p}f$ and $%
\mathcal{S}_{p}f,$ which are defined by (\ref{121}), (\ref{122}) and (\ref{123}), respectively.

\subsection*{Acknowledgement.} The research of \'A. Baricz was supported by a research grant
of the Romanian National Authority for Scientific Research, CNCS-UEFISCDI, project number PN-II-RU-TE-2012-3-0190/2014. The research of E. Deniz was supported by the Commission for the Scientific Research Projects of Kafkas Univertsity, project number 2012-FEF-30. The research of M. \c{C}a\u{g}lar and H. Orhan was supported by the Atat\"urk University Rectorship under The Scientific and Research Project of
Atat\"urk University, project number 2012/173.


\begin{thebibliography}{99}

\bibitem{An} S. Andr\'as and \'A. Baricz, Monotonicity property of generalized
and normalized Bessel functions of complex order, Complex Var. Elliptic
Equ., 54(7) (2009), 689--696.

\bibitem{Ba1} \'A. Baricz, Geometric properties of generalized Bessel
functions, Publ. Math. Debrecen, 73 (2008), 155--178.

\bibitem{Ba2} \'A. Baricz, Generalized Bessel functions of the first kind,
Lecture Notes in Mathematics, vol. 1994, Springer Verlag, Berlin, 2010.

\bibitem{Ba3} \'A. Baricz and S. Ponnusamy, Starlikeness and covexity of
generalized Bessel function, Integral Transform Spec. Funct., 21(9) (2010),
641--651.

\bibitem{De} E. Deniz, H. Orhan and H.M. Srivastava, Some sufficient
conditions for univalence of certain families of integral operators
involving generalized Bessel functions, Taiwansese J. Math., 15(2) (2011),
883--917.

\bibitem{De1} E. Deniz, Convexity of integral operators involving
generalized Bessel functions, Integral Transform Spec. Funct.,
24(3) (2013), 201--216.

\bibitem{Mil1} S.S. Miller and P.T. Mocanu, Differential subordinations, Marcel
Dekker, New York, 2000.

\bibitem{Mil2} S.S. Miller and P.T. Mocanu, Subordination of differential
superordinations, Complex Var. Theory Appl., 48(10) (2003), 815--826.

\bibitem{Mil3} S.S. Miller and P.T. Mocanu, Univalence of Gaussian and
confluent hypergeometric functions, Proc. Amer. Math. Soc., 110(2) (1990),
333--342.

\bibitem{Ow} S. Owa and H.M. Srivastava, Univalent and starlike generalized
hypergeometric functions, Canad. J. Math., 39(5) (1987), 1057--1077.

\bibitem{Po1} S. Ponnusamy and F. Ronning, Geometric properties for
convolutions of hypergeometric functions and functions with the derivative
in a halfplane, Integral Transform. Spec. Funct., 8 (1999), 121--138.

\bibitem{Po2} S. Ponnusamy and M. Vuorinen, Univalence and convexity
properties for confluent hypergeometric functions, Complex Var. Theory
Appl., 36(1) (1998), 73--97.

\bibitem{Po3} S. Ponnusamy and M. Vuorinen, Univalence and convexity
properties for Gaussian hypergeometric functions, Rocky Mountain J. Math.,
31(1) (2001), 327--353.

\bibitem{Pr} J.K. Prajapat, Certain geometric properties of normalized
Bessel functions, Appl. Math. Lett., (24) (2011), 2133--2139.

\bibitem{szasz} R. Sz\'asz, On starlikeness of Bessel functions of the first kind, In: Proceedings of the 8th Joint Conference on Mathematics
and Computer Science, Kom\'arno, Slovakia, 2010, 9pp.

\bibitem{szasz2} R. Sz\'asz and P.A. Kup\'an, About the univalence of the Bessel functions, Stud. Univ. Babe\c{s}-Bolyai Math., 54(1) (2009), 127--132.

\bibitem{Wa} G.N. Watson, A Treatise on the Theory of Bessel Functions, Cambridge University Press, Cambridge, 1944.
\end{thebibliography}
\end{document}